\documentclass[12pt]{amsart}
\usepackage{amsmath,amssymb}
\usepackage{comment}

\allowdisplaybreaks[1]

\title[On a characterization of unbounded homogeneous domains]
{On a characterization of unbounded homogeneous domains with a boundary of light cone type}

\author{Jun-ichi Mukuno}
\author{Yoshikazu Nagata}

\address{Graduate School of Mathematics, Nagoya University,
Furocho, Chikusaku, Nagoya 464-8602, Japan}

\email{m08043e@math.nagoya-u.ac.jp}
\email{m10035y@math.nagoya-u.ac.jp}

\keywords{automorphism, unbounded homogeneous domain}

\subjclass[2010]{11M32,11M06}

\theoremstyle{plain}
\newtheorem{theorem}{Theorem}
\newtheorem{lemma}{Lemma}

\newtheorem{claim}{Claim}

\theoremstyle{remark}
\newtheorem{remark}{Remark}

\numberwithin{theorem}{section}
\numberwithin{lemma}{section}
\numberwithin{corollary}{section}
\numberwithin{remark}{section}
\numberwithin{equation}{section}
\numberwithin{claim}{section}

\begin{document}
\maketitle

\begin{abstract}
We determine the automorphism groups of unbounded homogeneous domains
with boundaries of light cone type. Furthermore we present the group-theoretic 
characterization of the domain.
As a corollary we prove the non-existence of compact quotients of a 
unbounded homogeneous domain.
We also give a counterexample of the  
characterization.
\end{abstract}

\section*{Introduction}
The group-theoretic characterization problem of complex manifolds
asks whether 
the automorphism group of a complex manifold
determines only one biholomorphic equivalence class of complex manifolds.
Namely, for a given complex manifold,
it is discussed whether complex manifolds whose automorphism groups are isomorphic
to the automorphism group of the given complex manifold as topological
groups, are biholomorphic to each other.
Since there are many complex manifolds whose 
automorphism groups are trivial, 
this characterization problem has no sense for such manifolds.
Hence let us restrict our attention only to homogeneous complex manifolds,
in particular, homogeneous domains in the complex euclidean spaces.

By H.Cartan, it was shown that the 
automorphim groups of bounded domains have Lie group structures, 
and this result leads us to various studies of bounded homogeneous domains,
e.g. normal j-algebra of automorphim groups (see\cite{Pyateskii}).
Since normal j-algebras determine bounded homogeneous domains
with 1-1 correspondence,
one can say, 
that
automorphism groups characterize bounded homogeneous domains in this category.

For unbounded homogeneous domains, in contrast to the bounded domains,
automorphim groups are, in general, not (finite dimensional) Lie groups, 
and we have not obtained general theory
of automorphim groups and the characterization theorem.
Therefore any unbounded homogeneous domain is of interest, and some important cases are studied by Shimizu and Kodama \cite{KS1}, \cite{KS},
Byun, Kodama and Shimizu \cite{BKS}, etc.

In this paper, 
we proceed with a further example using Kodama and Shimizu's method in \cite{KS},
and some counterexample of group-theoretic characterization.
In order to describe our results,
let us fix notations here.
Let $\Omega$ be a complex manifold.
An {\it automorphism} of $\Omega$ means a biholomorphic mapping of $\Omega$ onto itself. 
We denote by $\mathrm{Aut}(\Omega)$ 
the group of all automorphisms of
$\Omega$ equipped with the compact-open topology.
$\Omega$ is called {\it homogeneous} if $\mathrm{Aut}(\Omega)$
acts transitively on $\Omega$.
The purpose of our paper is that
we determine
the automorphim group of the unbounded domain
\begin{equation*}
D^{n,1}=
\{ (z_0, \cdots, z_n) \in \mathbb{C}^{n+1} : -|z_{0}|^{2} + |z_{1}|^{2}+ \cdots + |z_{n}|^{2} > 0 \},
\end{equation*}
and give the characterization theorem of $D^{n,1}$
by the automorphim group $\mathrm{Aut}(D^{n,1})$.
$D^{n,1}$ is analogous to the de Sitter space
\begin{equation*}
\{(x_1, \cdots, x_n) \in \mathbb{R}^n : - x_{1}^{2} + x_{2}^{2} + \cdots x_{n}^{2} =1 \}.
\end{equation*}
The de Sitter space has a well-known property called
the Calabi-Markus phenomenon,
that is,
any isometry subgroups which acts properly discontinuously on the de Sitter space
are finite\cite{CM}. This phenomenon implies that the de Sitter space
has no compact quotient.
It is interesting whether similar results occur in other geometry.
We will study subgroups of the automorphism group $\mathrm{Aut}(D^{n,1})$
which acts properly discontinuously on $D^{n,1}$
and prove the non-existence of compact quotients of $D^{n,1}$.
It is not the precise Calabi-Markus phenomenon,
but a rigid phenomenon.
For these purposes,
we also need to consider the domain
\begin{equation*}
C^{n,1}=\{ (z_0, \cdots, z_n) \in \mathbb{C}^{n+1} : -|z_{0}|^{2} + |z_{1}|^{2}+ \cdots + |z_{n}|^{2} < 0 \},
\end{equation*}
the exterior of $D^{n,1}$ in $\mathbb{C}^{n+1}$.
To describe the automorphism groups $\mathrm{Aut}(D^{n,1})$ and $\mathrm{Aut}(C^{n,1})$,
put 
\begin{equation*}
GU(n,1)=\{A \in GL(n+1, \mathbb{C}):A^*JA = \nu(A)J, \mathrm{for\,some}\, \nu(A) \in \mathbb{R}_{>0} \},
\end{equation*}
where 
$J=
\begin{pmatrix}
-1 &   0\\
0 &  E_n
\end{pmatrix}
$.
Consider $\mathbb{C}^*$ 
as a subgroup of $GU(n,1)$:  
\begin{equation*}
\mathbb{C}^* \simeq \{
\begin{pmatrix}
\alpha & & \\
&  \ddots &  \\
& & \alpha
\end{pmatrix}:
 \alpha \in \mathbb{C}^* \} \subset GU(n,1).
\end{equation*}
Since $U(n,1)=\{A^*JA = J\} \subset GU(n,1)$ acts transitively on each level sets of
$-|z_{0}|^{2} + |z_{1}|^{2}+ \cdots + |z_{n}|^{2} (\neq 0)$,
and $\mathbb{C}^*$ acts on $D^{n,1}$ and $C^{n,1}$,
$GU(n,1)$ is a subgroup of the automorphism groups of these two domains 
$D^{n,1}$ and $C^{n,1}$.
It can be easily seen that $C^{n,1}$ and $D^{n,1}$ are homogeneous.
Now we state our main results.

\let\temp\thetheorem
\renewcommand{\thetheorem}{\ref{2}}
\begin{theorem}  
$\mathrm{Aut}(D^{n,1}) = GU(n,1)$.
\end{theorem}
\let\thetheorem\temp
\addtocounter{theorem}{-1}

We give the group-theoretic characterization thorem of $D^{n,1}$
in the class of complex manifolds contained in Stein manifolds.

\let\temp\thetheorem
\renewcommand{\thetheorem}{\ref{5}}
\begin{theorem} 
Let M be a connected complex manifold of dimension n+1 that is holomorphically separable
and admits a smooth envelope of holomorphy. 
Assume that $\mathrm{Aut}(M)$ is isomorphic to $\mathrm{Aut}(D^{n,1})$ as topological groups.
Then M is biholomorphic to $D^{n,1}$.
\end{theorem}
\let\thetheorem\temp
\addtocounter{theorem}{-1}

For the domain $C^{n,1}$, the characterization theorem was shown by 
Byun, Kodama  and Shimizu \cite{KS1}(see also the remark before Theorem \ref{BKS}). 

Our paper organizes as follows.
In Section 1,
first we will prepare notion of Reinhardt domains and 
Kodama-Shimizu's generalized standardization theorem,
which is the key lemma for our theorem.
To determine $\mathrm{Aut}(D^{n,1})$ we need an explicit form of 
the automorphism group $\mathrm{Aut}(C^{n,1})$. 
In Section 2 we determine $\mathrm{Aut}(C^{n,1})$.
We determine the automorphism groups of $D^{n,1}$
in Section 3.
We will show the non-existence of compact quotients of $D^{n,1}$ in Section 4,
using the Calabi-Markus phenomenon.
In Section 5
we prove the characterization theorem of 
$D^{n,1}$ by its automorphism group $\mathrm{Aut}(D^{n,1})$.
In Section 6, we construct a counterexample of 
the group-theoretic characterization of unbounded homogeneous domains.

\let\temp\thetheorem
\renewcommand{\thetheorem}{\ref{fth}}
\begin{theorem}
There exist unbounded homogeneous domains in $\mathbb{C}^n, n \geq 5$ which are not biholomorphically equivalent,
while its automorphism groups are isomorphic.
\end{theorem}
\let\thetheorem\temp
\addtocounter{theorem}{-1}

\section{Preliminary}
In order to establish terminology and notation, we recall some basic facts
about Reinhardt domains,
following Kodama and Shimizu \cite{KS1}\cite{KS} for convenience.

Let $G$ be a Lie group and
$\Omega$ a domain in $\mathbb{C}^{n}$.
Consider a
continuous group homomorphism $\rho : G \longrightarrow \mathrm{Aut}(\Omega)$. 
Then the mapping
\begin{equation*} 
G \times \Omega \ni (g, x) \longmapsto (\rho(g))(x) \in \Omega
\end{equation*}
is continuous, and in fact $C^{\omega}$.
we say that $G$ acts on $\Omega$
as a Lie transformation group through ¦Ñ. 
Let $T^{n} = (U(1))^{n}$, the $n$-dimensional
torus. 
$T^{n}$ acts holomorphically on $\mathbb{C}^n$ in the following standard manner:
\begin{equation*} 
T^n \times \mathbb{C}^n \ni (\alpha, z) \longmapsto 
\alpha \cdot z := (\alpha_1 z_1, \cdots, \alpha_n z_n) \in \mathbb{C}^n.
\end{equation*}
A {\it Reinhardt domain} $\Omega$ in $\mathbb{C}^{n}$ is, by definition,  
a domain which is stable under
this standard action of $T^{n}$. Namely, there exists a continuous map
$T^{n} \hookrightarrow \mathrm{Aut}(\Omega)$.
We denote the image of $T^{n}$ of this inclusion map by $T(\Omega)$.

Let $f$ be a holomorphic function on a Reinhardt domain $\Omega$,
then $f$ can be expanded uniquely into a Laurent series
\begin{eqnarray*}
f(z) = \sum_{\nu \in \mathbb{Z}^n} a_{\nu}z^{\nu},
\end{eqnarray*}
which converges absolutely and uniformly on any compact set in $\Omega$.
Here $z^{\nu} = z_1^{\nu_1} \cdots z_n^{\nu_n}$ for 
$\nu = (\nu_1, \cdots, \nu_n) \in \mathbb{Z}^n$.

$(\mathbb{C}^*)^n$ acts holomorphically on $\mathbb{C}^n$ as follows: 
\begin{equation*} 
(\mathbb{C}^*)^n \times \mathbb{C}^n \ni 
((\alpha_1, \cdots, \alpha_n), (z_1, \cdots, z_n)) \longmapsto 
(\alpha_1 z_1, \cdots, \alpha_n z_n) \in  \mathbb{C}^n.
\end{equation*}
We denote by $\Pi(\mathbb{C}^n)$ the group of all automorphisms of $\mathbb{C}^n$ of the form.
For a Reinhardt domain $\Omega$ in $\mathbb{C}^n$, we denote by $\Pi(\Omega)$
the subgroup of $\Pi(\mathbb{C}^n)$
consisting of all elements of $\Pi(\mathbb{C}^n)$ leaving $\Omega$ invariant. 
We need the following two lemmas to prove the characterization theorem.

\begin{lemma}[\cite{KS1}]\label{KS1}
Let $\Omega$ be a Reinhardt domain in $\mathbb{C}^n$ . 
Then $\Pi(\Omega)$ is the centralizer of $T(\Omega)$ in $\mathrm{Aut}(\Omega)$.
\end{lemma}

\begin{lemma}[Generalized Standardization Theorem\cite{KS}]\label{4}
Let M be a connected complex manifold of dimension $n$ that is holomorphically separable
and admits a smooth envelope of holomorphy, 
and let K be a connected compact Lie group of rank $n$.
Assume that an injective continuous group homomorphism 
$\rho$ of K into $\mathrm{Aut}(\Omega)$ is given.
Then there exists a biholomorphic map F of M onto a Reinhardt domain 
$\Omega$ in $\mathbb{C}^{n}$ such that
\begin{equation*} 
F\rho(K)F^{-1} = U(n_1) \times \cdots \times U(n_s) \subset \mathrm{Aut}(\Omega),
\end{equation*}
where $\sum_{j=1}^s n_j = n$.
\end{lemma}

\section{The automorhpsim group of $C^{n,1}$}

In this section, we consider the automorphism group $\mathrm{Aut}(C^{n,1})$ of the domain
\begin{equation*}
C^{n,1}=\{ (z_0, \ldots, z_n) \in \mathbb{C}^{n+1} : -|z_{0}|^{2} + |z_{1}|^{2}+ \cdots + |z_{n}|^{2} < 0 
\}.
\end{equation*}   

 \begin{theorem} \label{1}
  For 
 $f=(f_{0}, f_{1}, \ldots, f_{n}) \in \mathrm{Aut}(C^{n,1})$, we have
 \begin{align*}
 &f_{0}(z_{0}, z_{1}, \ldots, z_{n})=c\left(\frac{z_{1}}{z_{0}}, \frac{z_{2}}{z_{0}}, \ldots, 
 \frac{z_{n}}{z_{0}}\right) z_{0} 
 \ \mathrm{or} \ 
 c\left(\frac{z_{1}}{z_{0}}, \frac{z_{2}}{z_{0}}, \ldots, \frac{z_{n}}{z_{0}}\right) z_{0}^{-1}, \\
 \mathrm{and}\\
 &f_{i} (z_{0}, z_{1}, \ldots, z_{n}) =f_{0}(z_{0}, z_{1}, \ldots, z_{n}) 
 \frac{\sum_{j=0}^{n} a_{i j} z_{j}}{ \sum_{j=0}^{n} a_{0 j} z_{j}}, \,\,\,\,\mathrm{for}\,\, i=1, \ldots, n,
 \end{align*}
 where $c$ is a nowhere vanishing holomorphic function on 
 $\mathbb{B}^{n}$, and  the matrix $(a_{i j})_{0\leq i, j \leq n}$ is an element of $PU(n,1)$.  
 \end{theorem}

\begin{proof}

First we
remark that $C^{n,1}$ is biholomorphic to a product domain
$\mathbb{C}^* \times \mathbb{B}^{n}$.   
In fact, 
a biholomorphic map is given by
\begin{equation*}
\Psi : C^{n,1} \ni (z_{0}, z_{1}, \ldots, z_{n}) \longmapsto 
\left(z_{0}, \frac{z_{1}}{z_{0}}, \ldots, \frac{z_{n}}{z_{0}}\right) \in 
\mathbb{C}^* \times \mathbb{B}^{n}.  
\end{equation*}
Therefore, we consider the automorphism group of  $\mathbb{C}^* \times \mathbb{B}^{n}$.

Let $(w_{0}, w_{1}, \ldots, w_{n})$ be 
a coordinate of  $\mathbb{C}^* \times \mathbb{B}^{n}$,
and 
\begin{equation*}
\gamma=(\gamma_{0}, \gamma_{1}, \ldots, \gamma_{n})
\in \mathrm{Aut}(\mathbb{C}^* \times \mathbb{B}^{n}).
\end{equation*}
For fixed $(w_{1}, \ldots, w_{n}) \in \mathbb{B}^{n}$,  
$\gamma_{i}(\cdot, w_{1}, \ldots, w_{n})$ for $i=1, \ldots, n$ 
are bounded holomorphic functions on $\mathbb{C}^*$. 
Then, by the Riemann removable singularities theorem and the Liouville theorem, 
$\gamma_{i}(\cdot, w_{1}, \ldots, w_{n})$ for $i=1, \ldots, n$ 
are constant.  
Hence 
$\gamma_i$ $(i=1, \ldots, n)$ 
does not depend on  $w_{0}$.
In the same manner, we 
see that  for 
the inverse 
\begin{equation*}
\tau=(\tau_0, \tau_1,\ldots, \tau_n)=\gamma^{{-1}} 
\in \mathrm{Aut}(\mathbb{C}^* \times \mathbb{B}^{n})
\end{equation*}
of $\gamma$, 
the functions $\tau_i$ for $i=1, \ldots, n$ are independent of $w_{0}$.
It follows that 
\begin{equation*}
\overline{\gamma} := (\gamma_{1}, \gamma_{2}, \ldots, \gamma_{n}) \in \mathrm{Aut}(\mathbb{B}^{n}).
\end{equation*}
It is well-known (see\cite{Pyateskii}) that $\gamma \in \mathrm{Aut}(\mathbb{B}^{n})$ 
is a linear fractional transformation  of the form 
\begin{equation*}
\gamma_{i}( w_{1}, w_{2}, \ldots, w_{n})= \frac{a_{i0} + \sum_{j=1}^{n} a_{ij} w_{j}}{a_{00} + \sum_{j=1}^{n} a_{0j} w_{j}}, \,\,\,\,\, i=1,2, \ldots, n,
\end{equation*}
where  the matrix $(a_{i j})_{0\leq i, j \leq n}$ is an element of $PU(n,1)$.

Next we consider $\gamma_0$ of $\gamma$ and $\tau_0$ of $\tau$.
By regarding $\overline{\gamma}$ with 
the standard action of $ \mathrm{Aut}(\mathbb{B}^{n})$ on  $\mathbb{C}^* \times \mathbb{B}^{n}$,     
 we obtain a biholomorphic map
 \begin{equation*} 
 \gamma \circ \overline{\gamma}^{-1}(w_{0}, w_{1}, w_{2}, \ldots, w_{n})=
 (\gamma_{0}(w_{0}, \overline{\gamma}^{-1}(w_{1}, w_{2}, \ldots, w_{n})), 
 w_{1}, w_{2}, \ldots, w_{n})
 \end{equation*}
 on $\mathbb{C}^* \times \mathbb{B}^{n}$.
Thus for fixed $(w_{1}, w_{2}, \ldots, w_{n}) \in \mathbb{B}^{n}$, $\gamma_{0}$ is bijective 
on $\mathbb{C}^*$ with respect to $w_{0}$, and $\tau_0(w_{0}, \overline{\gamma}(w_{1}, 
w_{2}, \ldots, w_{n}))$ is its inverse.
Since  $\mathrm{Aut}(\mathbb{C}^*) = \{cw, cw^{-1}: c \in \mathbb{C}^* \}$,
 we have  $\gamma_{0}= c(w_{1},w_{2}, \ldots, w_{n}) w_{0}$ or  
 $c(w_{1},w_{2}, \ldots, w_{n}) w_{0}^{-1}$, where $c(w_{1},w_{2}, \ldots, w_{n})$ 
 is a nowhere vanishing holomorphic function on 
 $\mathbb{B}^{n}$.  

Since $\Psi ^{-1}\mathrm{Aut}(\mathbb{C}^* \times \mathbb{B}^{n})\Psi  = 
 \mathrm{Aut}(C^{n,1})$, we have shown the theorem. 
 
\end{proof}

\bigskip

We remark that the group-theoretic characterization of the domain 
$\mathbb{C}^* \times \mathbb{B}^{n}$ are proven
by Byun, Kodama and Shimizu\cite{BKS},
and in the paper more general domains are treated.

\begin{theorem}[J.Byun, A.Kodama and S.Shimizu \cite{BKS}] \label{BKS}
Let M be a connected complex manifold of dimension n+1 that is holomorphically separable
and admits a smooth envelope of holomorphy. 
Assume that $\mathrm{Aut}(M)$ is isomorphic to $\mathrm{Aut}(C^{n,1})$ as topological groups.
Then M is biholomorphic to $C^{n,1}$.
\end{theorem}

\section{The automorphism Group of $D^{n,1}$}

In this section, we determine the automorphism group $\mathrm{Aut}(D^{n,1})$ of the domain
\begin{equation*}
D^{n,1}=\{ (z_0, \ldots, z_n) \in \mathbb{C}^{n+1} : -|z_{0}|^{2} + |z_{1}|^{2}+ \cdots + |z_{n}|^{2} > 0 
\},
\end{equation*}  
which is the exterior of $C^{n,1}$.
We assume $n>1$. 
We show the following theorem using Theorem \ref{1} in the previous section.

\begin{theorem} \label{2}
$\mathrm{Aut}(D^{n,1}) = GU(n,1)$ for $n>1$.
\end{theorem}

\begin{proof}
Let 
$ f=(f_0, f_1, \ldots, f_n)\in \mathrm{Aut}(D^{n,1})$.
If $z'_0 \in \mathbb{C}$ is fixed, then the holomorphic functions $f_i(z'_0, \cdot, \cdots)$
for $i=0, \ldots, n$, on 
$D^{n,1} \cap \{ z_0 = z'_0 \}$ extend holomorphically to 
$\mathbb{C}^{n+1} \cap \{ z_0 = z'_0 \}$ by the Hartogs theorem.
Hence, when $z_0$ varies, we obtain an extended holomorphic map $f : \mathbb{C}^{n+1} \longrightarrow \mathbb{C}^{n+1}$ such that 
$f|_{D^{n,1}} \in \mathrm{Aut}(D^{n,1})$.
The same consideration for $f^{-1} \in \mathrm{Aut}(D^{n,1})$ 
shows that there exists a holomorphic map $g : \mathbb{C}^{n+1} \longrightarrow \mathbb{C}^{n+1}$, 
such that
$g|_{D^{n,1}} = f^{-1}$.
Since $g\circ f = \mathrm{id}$ and $f\circ g = \mathrm{id}$ on $D^{n,1}$, 
the uniqueness of analytic continuation shows that
$g\circ f = \mathrm{id}$ and $f\circ g = \mathrm{id}$ on $\mathbb{C}^{n+1}$.
Hence $f \in \mathrm{Aut(\mathbb{C}^{n+1})}$, so that 
$\mathrm{Aut}(D^{n,1}) \subset \mathrm{Aut}(\mathbb{C}^{n+1})$.

Now we know that $f|_{C^{n,1}} \in \mathrm{Aut}(C^{n,1})$.
By Theorem \ref{1} of the previous section, we have
\begin{equation*}
f_{0}(z_{0}, z_{1}, z_{2}, \ldots, z_{n})=c\left(\frac{z_{1}}{z_{0}}, \frac{z_{2}}{z_{0}}, \ldots, 
\frac{z_{n}}{z_{0}}\right) z_{0}^{\pm 1},
\end{equation*}
and
\begin{equation*}
f_{i} (z_{0}, z_{1}, z_{2}, \ldots, z_{n}) =f_{0}(z_{0}, z_{1}, z_{2}, \ldots, z_{n}) 
\gamma_i\left(\frac{z_{1}}{z_{0}}, \frac{z_{2}}{z_{0}}, \ldots, \frac{z_{n}}{z_{0}}\right)
\end{equation*}
for $i=1, \ldots, n$, where $c$ is a nowhere vanishing holomorphic function on $\mathbb{B}^n$ 
and 
\begin{equation*}
\gamma_i(w_1, \ldots, w_n) =
\frac{a_{i0} + \sum_{j=0}^{n} a_{ij}w_j}{a_{00} + \sum_{j=0}^{n} a_{0j}w_j}.
\end{equation*}

If we have
\begin{eqnarray*}
f_0(z) = c\left(\frac{z_{1}}{z_{0}}, \frac{z_{2}}{z_{0}}, \ldots, 
 \frac{z_{n}}{z_{0}}\right) z_{0}^{-1}, 
\end{eqnarray*}  
considering 
the Taylor expansion of $c$ at the origin, we see that $f_0$ is not holomorphic at $z_0 = 0$, 
which contradicts the fact that $f_0$ is an entire holomorphic function.
Thus we have
\begin{eqnarray*}
f_0(z) = c\left(\frac{z_{1}}{z_{0}}, \frac{z_{2}}{z_{0}}, \ldots, 
 \frac{z_{n}}{z_{0}}\right) z_{0}.
\end{eqnarray*}  
Then the entire functions $f_i \,\, (i=1, \ldots, n)$ are expressed as
\begin{eqnarray*} 
f_i(z_{0}, \ldots, z_{n}) 
&=&  \gamma_i
\left(\frac{z_1}{z_0}, \ldots, \frac{z_n}{z_0}\right) 
c\left(\frac{z_1}{z_0}, \ldots, \frac{z_n}{z_0}\right)z_0\\
&=&
(a_{i0}z_0 + \sum_{j=0}^{n} a_{ij}z_j)
c\left(\frac{z_1}{z_0}, \ldots, \frac{z_n}{z_0}\right) 
\left/ \left(a_{00} + \sum_{j=0}^{n} a_{0j}\frac{z_j}{z_0}\right)\right.,
\end{eqnarray*} 
and hence
$c(w_1, \ldots, w_n)$ must be divided by
$a_{00} + \sum_{j=0}^{n} a_{0j}w_j$.
We now write 
\begin{eqnarray*} 
c(w_1, \ldots, w_n) = (a_{00} + \sum_{j=0}^{n} a_{0j}w_j) \tilde{c}(w_1, \ldots, w_n), 
\end{eqnarray*} 
then 
\begin{eqnarray*} 
f_i(z_{0}, \ldots, z_{n})
=
(a_{i0}z_0 + \sum_{j=0}^{n} a_{ij}z_j) \tilde{c}\left(\frac{z_1}{z_0}, \ldots, \frac{z_n}{z_0}\right).
\end{eqnarray*} 
Since $\tilde{c}(\frac{z_1}{z_0}, \ldots, \frac{z_n}{z_0})$ is holomorphic near $z_0 = 0$,
the holomorphic function $\tilde{c}$ must be a non-zero constant $C$.
Consequently, we obtain
\begin{eqnarray*}
f(z_0, \ldots, z_n) = \left(C\sum_{j=0}^n a_{0j}z_j, \ldots, C\sum_{j=0}^n a_{nj}z_j\right).
\end{eqnarray*}
Thus we have shown the theorem.

\end{proof}

\section{The non-existence of compact quotients of $D^{n,1}$}

In this section, 
we prove the following theorem: 
\begin{theorem}  \label{3}
$D^{n,1}$, for $n>1$, has no compact quotient by a discrete subgroup of $\mathrm{Aut}(D^{n,1})$ acting properly discontinuously.
\end{theorem}
We remark that $C^{n,1}$ has compact quotients
since $\mathbb{B}^{n}$ and $\mathbb{C}^*$ has compact quotients. 
Recall the following result called the Calabi--Markus phenomenon:   
 \begin{lemma}[Calabi--Markus\cite{CM}, Wolf\cite{Wolf}]~\label{CMW}
Let $\Gamma$ be  a subgoup of $O(p,q+1)$ acting properly disconcinuously 
on 
\begin{eqnarray*} 
\{(x_1, \ldots, x_{p},  x_{p+1}, \ldots, x_{p+q+1}) 
\in \mathbb{R}^{p+q+1}:- x_{1}^{2}- \cdots - x_{p}^{2} +x_{p+1}^{2} +
\cdots + x_{p+q+1}^2=1 \},
\end{eqnarray*} 
where $1 <p \leq q $. 
Then $\Gamma$ is finite.  
\end{lemma}

\begin{proof}
From  Theorem~\ref{2}, we know that 
$\mathrm{Aut}(D^{n,1}) = GU(n,1) = \mathbb{R}_{> 0} \times U(n,1)$, 
which acts on the complex euclidean space as linear transformations. 
We regard $\mathbb{R}_{> 0} \times U(n,1)$  as  a subgroup of $\mathbb{R}_{> 0} \times O(2n,2)$. 

Suppose that there exists a discrete subgroup 
\begin{eqnarray*} 
\Gamma = \{ f_{m} \}_{m=1}^{\infty} \subset \mathbb{R}_{>0} \times O(2n,2)
\end{eqnarray*} 
such that 
$\Gamma$ acts properly discontinuously  on $D^{n,1}$ and
that  the quotient $D^{n,1}/\Gamma$ is compact.  
By Selberg's lemma, we may assume without loss of 
generality that $\Gamma$ is torsion free. 
Set $f_{m}=(r_{m}, T_{m})$, where $r_{m}\in \mathbb{R}_{>0} $ and  $T_{m} \in O(2n,2)$. 
It is clear that $\Gamma$ is not included in $O(2n,2)$ by Lemma~\ref{CMW}.  
We consider two cases.

First  we consider the case where there exists the minimum of the set $\{ r_{m}  |  1 < r_{m}  \}$.   
We denote the minimum by $R$:
\begin{eqnarray*} 
R = \min\{ r_{m}  |  1 < r_{m}  \}.
\end{eqnarray*} 
Then we see that, for any $r_{m}$, there exists an integer $l$ such that  $r_{m} = R^{l}$. 
Therefore we can write 
\begin{eqnarray*} 
\Gamma =\{f_{l, k}=(R^{l}, T_{l, k})  \}_{l \in \mathbb{Z}, k \in \mathbb{N} }
\end{eqnarray*} 
by changing the indexes.  
Put $\Gamma_{0}=\{ f_{0,k}\}$, a subgroup of $O(2n,2)$. 
By Theorem~\ref{CMW}, it follows that $\Gamma_{0}$ is a finite group. 
Since $\Gamma_{0}$ is torsion free, $\Gamma_{0}=\{ \mathrm{id} \}$.  
Therefore, $\Gamma $ is the group generated by  the element $(R, T)  \in \Gamma$.
Hence we see that $D^{n,1}/\Gamma$ is not compact.

Next  we consider the case where there does not 
exist the minimum of the set $\{ r_{m}  |  1 < r_{m}  \}$.  
Let $R^{\prime}$ be the inifimum of  the set $\{ r_{m}  |  1 < r_{m}  \}$:
\begin{eqnarray*} 
R' = \inf\{ r_{m}  |  1 < r_{m}  \}.
\end{eqnarray*} 
Then, for any $\epsilon>0$,
by arranging the indexes of the elements of $\Gamma$, 
we can take an infinite distinct sequence 
\begin{eqnarray*} 
R^{\prime}+\epsilon >r_{1}> r_{2}>r_{3}> \cdots > r_{m}> \cdots > R^{\prime}.  
\end{eqnarray*} 
Let 
\begin{eqnarray*} 
\Pi = \{z_{0}=0\} \subset \mathbb{C}^{n+1}
\end{eqnarray*} 
and 
\begin{eqnarray*} 
K = \{z_{0}=0, 1 \leq |z_{1}|^{2} + |z_{2}|^{2} + \cdots |z_{n}|^{2} \leq (R^{\prime}+\epsilon)^{2} +1 \}
 \subset \mathbb{C}^{n+1}.
 \end{eqnarray*} 
It is clear that $K$ is compact in $D^{n,1}$.  
Let $\gamma_{m}=(r_{m}, T_{m})$. 
 We can easily see that $\gamma_{m}(\Pi) \cap \Pi$ contains a nontrivial linear subspace 
 by the dimension formula of linear map. 
Then there exist $v_{m} \in \gamma_{m}(\Pi) \cap \Pi$ and $w_{m} \in \Pi$ such that  $v_{m}=\gamma_{m}(w_{m})$
and 
that  $|w_{m}| = 1$.  
Note that $w_{m} \in  K$.  
 We see that 
 $|v_{m}| = r_{m} |w_{m}| = r_{m} \leq R^{\prime}+\epsilon$, since $v_{m} \in \Pi$, and thus 
 $v_{m} \in K$. 
 We obtain that $\gamma_{m}(K) \cap K \not= \emptyset$ for any $m \geq 1$. 
However this is a contradiction since $\Gamma$ acts on properly discontinuously.  
The proof is complete. 
\end{proof}

\section{A characterization of $D^{n,1}$ by its automorphism group}

We record first some results, which will be used in the proof of
the main theorem several times.
\begin{lemma}  \label{notLie}
Let $D$ and $D'$ be domains in $\mathbb{C}^n$.
Then the automorphism groups of domains
$\mathbb{C} \times D$, $\mathbb{C}^* \times D'$
and
$(\mathbb{C} \times D) \cup (\mathbb{C}^* \times D')$
are not Lie groups.
\end{lemma}

\begin{proof}
For any nowhere vanishing holomorphic function $u$ on $\mathbb{C}^n$,
$f(z) = (u(z_1, \ldots, z_n)z_0, z_1, \ldots, z_n)$ is an automorphism on each domain.
Indeed, the inverse is given by $g(z) = (u(z_1, \ldots, z_n)^{-1}z_0, z_1, \ldots, z_n)$.
Thus the automorphism groups of these domains have no Lie group structures.
\end{proof}

\begin{lemma}  \label{notsimple}
Let $p, q, k$ be non-negative integers and $p + q \geq 2$. For $p + q > k$, any Lie group homomorphism
\begin{equation*} 
\rho : SU(p,q) \longrightarrow GL(k, \mathbb{C})
\end{equation*}
is trivial.
\end{lemma}

\begin{proof}
Put $n = p + q$.
It is enough to show that the Lie algebra homomorphism
\begin{equation*} 
d\rho : \mathfrak{su}(p,q) \longrightarrow \mathfrak{gl}(k, \mathbb{C})
\end{equation*}
is trivial.
Consider its complex linear extension
\begin{equation*} 
d\rho_\mathbb{C}  : \mathfrak{su}(p,q) \otimes_{\mathbb{R}} \mathbb{C} 
\longrightarrow \mathfrak{gl}(k, \mathbb{C}).
\end{equation*}
Since $\mathfrak{su}(p,q) \otimes_{\mathbb{R}} \mathbb{C} = \mathfrak{sl}(n, \mathbb{C})$
and $\mathfrak{sl}(n, \mathbb{C})$ is a simple Lie algebra,
$d\rho_\mathbb{C}$ is injective or trivial.
On the other hand, $\dim_{\mathbb{C}} \mathfrak{su}(p,q) \otimes_{\mathbb{R}} \mathbb{C} 
= n^2 - 1 > k^2 = \dim_{\mathbb{C}} \mathfrak{gl}(k, \mathbb{C})$.
Thus $d\rho_\mathbb{C}$ must be trivial,
and so is $d\rho$.
\end{proof}

Now we prove the following main theorem.

\begin{theorem}  \label{5}
Let M be a connected complex manifold of dimension n+1 that is holomorphically separable
and admits a smooth envelope of holomorphy. 
Assume that $\mathrm{Aut}(M)$ is isomorphic to $\mathrm{Aut}(D^{n,1})=GU(n,1)$ as topological groups.
Then M is biholomorphic to $D^{n,1}$.
\end{theorem}

\begin{proof}
Denote by $\rho_0 : GU(n,1) \longrightarrow \mathrm{Aut}(M)$ a topological group isomorphism.
Let us consider $U(1) \times U(n)$ as a matrix subgroup of 
$GU(n,1)$ in the natural way,
and identify $U(n)$ with $\{1\} \times U(n)$.
Then, by Lemma \ref{4}, there is a biholomorphic map F from M onto a Reinhardt domain 
$\Omega$ in $\mathbb{C}^{n+1}$ such that
\begin{equation*} 
F\rho_0(U(1) \times U(n))F^{-1} 
= U(n_1) \times \cdots \times U(n_s) \subset \mathrm{Aut}(\Omega),
\end{equation*}
where $\sum_{j=1}^s n_j = n+1$.
Then, after a permutation of coordinates if we need, we may assume
$F\rho_0(U(1) \times U(n))F^{-1} = U(1) \times U(n)$.
We define an isomorphism
\begin{equation*} 
\rho: GU(n,1) \longrightarrow \mathrm{Aut}(\Omega)
\end{equation*}
by
$\rho(g) := F \circ \rho_0(g) \circ F^{-1}$.
We will prove that $\Omega$ is biholomorphic to $D^{n,1}$.

Put
\begin{equation*} 
T_{1,n}=
\left\{
\begin{pmatrix}
u_1 &  \\
  &  u_2E_n \\
\end{pmatrix}:
 u_1, u_2 \in U(1)
 \right\} \subset GU(n,1).
\end{equation*} 
Since $T_{1,n}$
is the center of the group $U(1) \times U(n)$,
we have $\rho(T_{1,n}) = T_{1,n} \subset \mathrm{Aut}(\Omega)$.
Consider $\mathbb{C}^*$ 
as a subgroup of $GU(n,1)$. 
So $\mathbb{C}^*$ represents center of $GU(n,1)$.
Since $\rho(\mathbb{C}^*)$ is 
commutative with $T^{n+1}$,
Lemma \ref{KS1} tells us that $\rho(\mathbb{C}^*) \subset \Pi(\Omega)$,
that is, $\rho(\mathbb{C}^*)$ is represented by diagonal matrices.
Furthermore, $\rho(\mathbb{C}^*)$ commutes with $\rho(U(1) \times U(n)) = U(1) \times U(n)$,
so that we have 
\begin{eqnarray*}
\rho\left(e^{2\pi i(s+it)}\right) =
\begin{pmatrix}
e^{2\pi i\{a_1s+(b_1+ic_1)t\}} & \\
&  e^{2\pi i\{a_2s+(b_2+ic_2)t\}}E_n
\end{pmatrix}
\in \rho(\mathbb{C}^*),
\end{eqnarray*}
where
$s,t \in \mathbb{R}$, $a_1, a_2 \in \mathbb{Z}, b_1,b_2,c_1,c_2 \in \mathbb{R}$.
Since $\rho$ is injective,
$a_1, a_2$ are
relatively prime
and $(c_1, c_2) \neq (0, 0)$.
To consider the actions of $\rho(\mathbb{C}^*)$ and $U(1) \times U(n)$ 
on $\Omega$ together,
we put
\begin{equation*} 
G(U(1) \times U(n)) = 
\left\{e^{-2\pi t}
\begin{pmatrix}
u_0 & \\
&  U
\end{pmatrix}
\in GU(n,1) : t \in \mathbb{R}, u_0 \in U(1), U \in U(n)
\right\}.
\end{equation*}
Then we have
\begin{eqnarray*} 
G &:=&\rho(G(U(1) \times U(n))) \\
&{}=&
\left\{
\begin{pmatrix}
e^{-2\pi c_1t}u_0 & \\
&  e^{-2\pi c_2t}U
\end{pmatrix}
\in GL(n+1,\mathbb{C}) : t \in \mathbb{R}, u_0 \in U(1), U \in U(n)
\right\}.
\end{eqnarray*}
Note that $G$ is the centralizer
of $T_{1,n} = \rho(T_{1,n})$ in $\mathrm{Aut}(\Omega)$.

Let $f =(f_0, f_1, \ldots, f_n) \in \mathrm{Aut}(\Omega)\setminus G$ and consider 
its Laurent expansions:\\
\begin{eqnarray}
f_0(z_0, \ldots, z_n) 
&=& \sum_{\nu \in \mathbb{Z}^{n+1}} a^{(0)}_{\nu}z^{\nu},\label{eqf0} \\
f_i(z_0, \ldots, z_n) 
&=& \sum_{\nu \in \mathbb{Z}^{n+1}} a^{(i)}_{\nu}z^{\nu}, \,\,\,\, 1\leq i \leq n.\label{eqfi}
\end{eqnarray}
If $f$ is a linear map of the form
\begin{equation*}
\begin{pmatrix}
a^{(0)}_{(1,0,\ldots,0)} & 0 & \cdots & 0 &\\
0 & a^{(1)}_{(0,1,0,\ldots,0)} & \cdots & a^{(1)}_{(0,\ldots,0,1)}\\
 \vdots  & \vdots &  \ddots & \vdots\\
0 & a^{(n)}_{(0,1,0,\ldots,0)} & \cdots           & a^{(n)}_{(0,\ldots,0,1)}
\end{pmatrix}
\in GL(n+1, \mathbb{C}).
\end{equation*}
then $f$ commutes with $\rho(T_{1,n})$, which contradicts $f \notin G$.
Thus for any $f \in \mathrm{Aut}(\Omega) \setminus G$,
there exists $\nu \in \mathbb{Z}^{n+1} (\neq (1,0,\ldots,0))$
such that $a^{(0)}_{\nu} \neq 0$ in (\ref{eqf0}),
or 
there exists
$\nu \in \mathbb{Z}^{n+1} (\neq (0,1,0 \ldots,0)$, 
$\ldots$, $(0,0 \ldots,0,1))$
such that
$a^{(i)}_{\nu} \neq 0$ in (\ref{eqfi}) for some $1 \leq i \leq n$.

\bigskip
\begin{remark}\label{remt}
We remark here that, in $(\ref{eqf0})$ and $(\ref{eqfi})$, there are no negative degree terms 
of $z_1,\ldots,z_n$,
since $\Omega \cup \{z_i = 0\} \neq \emptyset$ for $1 \leq i \leq n$
by the $U(n)$-action on $\Omega$,
and since Laurent expansions are globally defined on $\Omega$.
Write $\nu=(\nu_0, \nu')=(\nu_0, \nu_1, \ldots, \nu_n)$
and $|\nu'|=\nu_1 + \cdots + \nu_n$.
Let us consider $\nu' \in \mathbb{Z}^n_{\geq 0}$
and put
\begin{eqnarray*}
{\sum_{\nu}}' = \sum_{\nu_0\in \mathbb{Z}, \nu' \in \mathbb{Z}^n_{\geq 0}}
\end{eqnarray*}
from now on.
\end{remark}

\bigskip

\begin{claim}\label{cla1}
$a_1a_2c_1c_2 \neq 0$, and $\lambda := c_2/c_1 = a_2/a_1 = \pm 1$.
\end{claim}

\begin{proof}
To prove the claim, we divide three cases.

\smallskip
{\bf Case (i)}: 
$c_1c_2 \neq 0$.\\
Since $\mathbb{C}^*$ is the center of $GU(n,1)$, 
it follows that, 
for $f \in \mathrm{Aut}(\Omega) \setminus G$, 
\begin{eqnarray*}
f \circ \rho(e^{2\pi i(s+it)}) = \rho(e^{2\pi i(s+it)}) \circ f.
\end{eqnarray*}
By $(\ref{eqf0})$ and $(\ref{eqfi})$,
this equation means
\begin{eqnarray*}
e^{2\pi i\{a_1s+(b_1+ic_1)t\}} {\sum_{\nu}}' a^{(0)}_{\nu}z^{\nu}
&=& 
{\sum_{\nu}}' a^{(0)}_{\nu}
(e^{2\pi i\{a_1s+(b_1+ic_1)t\}}z_0)^{\nu_0^{(0)}}
(e^{2\pi i\{a_2s+(b_2+ic_2)t\}}z')^{\nu'}\\
&=&
{\sum_{\nu}}' a^{(0)}_{\nu}
e^{2\pi i\{a_1s+(b_1+ic_1)t\}\nu_0^{(0)}} e^{2\pi i\{a_2s+(b_2+ic_2)t\}|\nu'|}
z^{\nu}
\end{eqnarray*}
and
\begin{eqnarray*}
e^{2\pi i\{a_2s+(b_2+ic_2)t\}} {\sum_{\nu}}' a^{(i)}_{\nu}z^{\nu}
&=&
{\sum_{\nu}}' a^{(i)}_{\nu}
(e^{2\pi i\{a_1s+(b_1+ic_1)t\}}z_0)^{\nu_0^{(i)}}
(e^{2\pi i\{a_2s+(b_2+ic_2)t\}}z')^{\nu'}\\
 &=&
{\sum_{\nu}}' a^{(i)}_{\nu}
e^{2\pi i\{a_1s+(b_1+ic_1)t\}\nu_0^{(i)}} e^{2\pi i\{a_2s+(b_2+ic_2)t\}|\nu'|} z^{\nu},
\end{eqnarray*}
for $1 \leq i \leq n$.
Thus for each $\nu \in \mathbb{Z}^{n+1}$, we have
\begin{eqnarray*}
e^{2\pi i\{a_1s+(b_1+ic_1)t\}} a^{(0)}_{\nu}
= e^{2\pi i\{a_1s+(b_1+ic_1)t\}\nu_0^{(0)}} e^{2\pi i\{a_2s+(b_2+ic_2)t\}|\nu'|} a^{(0)}_{\nu},
\end{eqnarray*}
and
\begin{eqnarray*}
e^{2\pi i\{a_2s+(b_2+ic_2)t\}} a^{(i)}_{\nu}
= e^{2\pi i\{a_1s+(b_1+ic_1)t\}\nu_0^{(i)}} e^{2\pi i\{a_2s+(b_2+ic_2)t\}|\nu'|} a^{(i)}_{\nu},
\end{eqnarray*}
for $1 \leq i \leq n$.
Therefore, if
$a_{\nu}^{(0)}\neq0$ for $\nu=(\nu_0^{(0)}, \nu')=(\nu_0^{(0)}, \nu_1^{(0)}, \ldots, \nu_n^{(0)})$,
we have
\begin{eqnarray}\label{ac0}
\hspace{-3cm}
\left\{\begin{array}{l}
a_1(\nu_0^{(0)}-1) + a_2(\nu_1^{(0)} + \cdots + \nu_n^{(0)}) = 0,\\
c_1(\nu_0^{(0)}-1) + c_2(\nu_1^{(0)} + \cdots + \nu_n^{(0)}) = 0.
\end{array}
\right.
\end{eqnarray}
Similarly, if $a_{\nu}^{(i)} \neq 0$
for $\nu=(\nu_0^{(i)}, \nu')=(\nu_0^{(i)}, \nu_1^{(i)}, \ldots, \nu_n^{(i)})$,
we have
\begin{eqnarray}\label{aci}
\hspace{-33mm}
\left\{\begin{array}{l}
a_1\nu_0^{(i)} + a_2(\nu_1^{(i)} + \cdots + \nu_n^{(i)}-1) = 0,\\
c_1\nu_0^{(i)} + c_2(\nu_1^{(i)} + \cdots + \nu_n^{(i)}-1) = 0,
\end{array}
\right.
\end{eqnarray}
for $1 \leq i \leq n$.

Suppose $a_{\nu}^{(0)}\neq0$ for some 
$\nu=(\nu_0^{(0)}, \nu_1^{(0)}, \ldots, \nu_n^{(0)}) \neq (1,0,\ldots,0)$.
Then by (\ref{ac0}) and the assumption $c_1c_2 \neq 0$
it follows that $\nu_0^{(0)}-1 \neq 0$ and $\nu_1^{(0)} + \cdots + \nu_n^{(0)} \neq 0$.
Hence $c_2/c_1 \in \mathbb{Q}$ and $(a_1,a_2) \neq (\pm1,0), (0,\pm1)$ by (\ref{ac0}).
On the other hand,
if $a_{\nu}^{(i)}\neq0$ for some 
$1 \leq i \leq n$ and $\nu=(\nu_0^{(i)}, \nu_1^{(i)}, \ldots, \nu_n^{(i)}) \neq (0,1,0 \ldots,0), 
\ldots, (0,0 \ldots,0,1)$,
then 
$\nu_0^{(i)} \neq 0$ and $\nu_1^{(i)} + \cdots + \nu_n^{(i)} -1 \neq 0$
by (\ref{aci}) and the assumption $c_1c_2 \neq 0$.
In this case, we also obtain $c_2/c_1 \in \mathbb{Q}$ and 
$(a_1,a_2) \neq (\pm1,0), (0,\pm1)$ by (\ref{aci}).
Consequently, we have
\begin{eqnarray*}
\lambda:= a_2/a_1 = c_2/c_1 \in \mathbb{Q}
\end{eqnarray*}
by (\ref{ac0}) or (\ref{aci}).

\bigskip

We now prove that $\lambda$ is a integer.
For the purpose, we assume $\lambda \notin \mathbb{Z}$, that is, $a_1 \neq \pm1$.
First we consider the case $\lambda < 0$.
Since 
$\nu_1^{(i)} + \cdots + \nu_n^{(i)} \geq 0$ for $0 \leq i \leq n$, 
we have $\nu_0^{(0)} \geq 1$ and $\nu_0^{(i)} \geq 0$ by (\ref{ac0}) and (\ref{aci}) .
Furthermore, the Laurent expansions of the components of $f \in \mathrm{Aut}(\Omega)$ are
\begin{eqnarray}\label{form0}
f_0(z_0, \ldots, z_n) 
= \sum_{k=0}^{\infty} \, {\sum_{|\nu'|=k|a_1|}}^{\hspace{-8pt}\prime}\,\,\,
a^{(0)}_{\nu'}z_0^{1+k|a_2|}(z')^{\nu'}
\end{eqnarray}
and
\begin{eqnarray}\label{formi}
f_i(z_0, \ldots, z_n) 
= \sum_{k=0}^{\infty} \, {\sum_{|\nu'|=1+k|a_1|}}^{{\hspace{-13pt}\prime}} \,\,\,
a^{(i)}_{\nu'}z_0^{k|a_2|}(z')^{\nu'}
\end{eqnarray}
for $1 \leq i \leq n$.
Here we have written $a^{(0)}_{\nu'}=a^{(0)}_{(1+k|a_2|,\nu')}$
and $a^{(i)}_{\nu'}=a^{(i)}_{(k|a_2|,\nu')}$,
and so as from now on.

We focus on the first degree terms of the Laurent expansions.
We put
\begin{eqnarray}\label{Pf}
Pf (z) := (a^{(0)}_{(1,0,\ldots,0)}z_0, {\sum_{|\nu'|=1}}' a^{(1)}_{\nu'}(z')^{\nu'}, \ldots,
{\sum_{|\nu'|=1}}' a^{(n)}_{\nu'}(z')^{\nu'}).
\end{eqnarray}
Then as a matrix we can write
\begin{eqnarray*}
Pf = 
\begin{pmatrix}
a^{(0)}_{(1,0,\ldots,0)} & 0 & \cdots & 0 &\\
0 & a^{(1)}_{(0,1,0,\ldots,0)} & \cdots & a^{(1)}_{(0,\ldots,0,1)}\\
 \vdots  & \vdots &  \ddots & \vdots\\
0 & a^{(n)}_{(0,1,0,\ldots,0)} & \cdots           & a^{(n)}_{(0,\ldots,0,1)}
\end{pmatrix}.
\end{eqnarray*}
Then it follows from (\ref{form0}) and (\ref{formi}) that
\begin{eqnarray*}\label{}
P(f \circ h) = Pf \circ Ph, \,\,\,{\rm and}\,\,\, P{\rm id} = {\rm id},
\end{eqnarray*}
where $h \in \mathrm{Aut}(\Omega)$,
and therefore 
\begin{eqnarray*}\label{}
Pf  \in GL(n+1, \mathbb{C})
\end{eqnarray*}
since $f$ is an automorphism.
Hence we have a representation of $GU(n,1)$ given by
\begin{eqnarray*}
GU(n,1) \ni g \longmapsto Pf \in GL(n+1,\mathbb{C}),
\end{eqnarray*}
where $f = \rho(g)$.
The restriction of this representation to the simple Lie group $SU(n,1)$
is nontrivial since $\rho(U(1) \times U(n)) = U(1) \times U(n)$.
However this contradicts Lemma \ref{notsimple}.
Thus it does not occur that $\lambda$ is a negative non-integer.

\bigskip

Next we consider the case $\lambda >0$ and $\lambda \not \in \mathbb{Z}$.
Then $\nu_0^{(0)} \leq 1$ and $\nu_0^{(i)} \leq 0$ by (\ref{ac0}) and (\ref{aci}) 
since $\nu_1^{(i)} + \cdots + \nu_n^{(i)} \geq 0$ for $0 \leq i \leq n$.
Furthermore, the Laurent expansions of $f$ are
\begin{eqnarray*}
f_0(z_0, \ldots, z_n) 
&=& \sum_{k=0}^{\infty} \, {\sum_{|\nu'|=k|a_1|}}^{\hspace{-8pt}\prime}\,\,\,
a^{(0)}_{\nu'}z_0^{1-k|a_2|}(z')^{\nu'}\\
&=& a^{(0)}_{(1,0,\ldots,0)}z_0 
+ {\sum_{|\nu'|=|a_1|}}^{\hspace{-5pt}\prime} \,\,a^{(0)}_{(1-|a_2|,\nu')}z_0^{1-|a_2|}(z')^{\nu'}\\
&&\,\,\,\,\,\,\,\,\,\,\,\,\,\,\,\,\,\,\,\,\,\,\,\,\,\,\,\,\,\,\,\,\,\,\,\,\,\,\,\,{\hspace{13pt}}
+ {\sum_{|\nu'|=2|a_1|}}^{\hspace{-8pt}\prime} \,\,a^{(0)}_{(1-2|a_2|,\nu')}z_0^{1-2|a_2|}(z')^{\nu'} 
+ \cdots,
\end{eqnarray*}
and
\begin{eqnarray*}
f_i(z_0, \ldots, z_n) 
&=& \sum_{k=0}^{\infty} \, {\sum_{|\nu'|=1+k|a_1|}}^{\hspace{-13pt}\prime}\,\,\, 
a^{(i)}_{\nu'}z_0^{-k|a_2|}(z')^{\nu'}\\
&=& {\sum_{|\nu'|=1}}^{\prime} \,\,a^{(i)}_{(0,\nu')}(z')^{\nu'} 
+ {\sum_{|\nu'|=1+|a_1|}}^{\hspace{-10pt}\prime} \,\,a^{(i)}_{(-|a_2|,\nu')}z_0^{-|a_2|}(z')^{\nu'}\\
&&\,\,\,\,\,\,\,\,\,\,\,\,\,\,\,\,\,\,\,\,\,\,\,\,\,\,\,\,\,\,\,\,\,\,\,\,\,\,\,\,{\hspace{13pt}}
+ {\sum_{|\nu'|=1+2|a_1|}}^{\hspace{-13pt}\prime} \,\,a^{(0)}_{(-2|a_2|,\nu')}z_0^{-2|a_2|}(z')^{\nu'} 
+ \cdots
\end{eqnarray*}
for $1\leq i \leq n$.
We claim that $a^{(0)}_{(1,0,\ldots,0)} \neq 0$.
Indeed, if $a^{(0)}_{(1,0,\ldots,0)} = 0$,
then $f(z_0, 0, \ldots, 0) = (0, \ldots, 0) \in \mathbb{C}^{n+1}$.
This contradicts that $f$ is an automorphism.
Take another $h \in \mathrm{Aut}(\Omega) \setminus G$
and put its Laurent expansions
\begin{eqnarray*}
h_0(z_0, \ldots, z_n) 
&=& \sum_{k=0}^{\infty} \, {\sum_{|\nu'|=k|a_1|}}^{\hspace{-8pt}\prime}\,\,\,
b^{(0)}_{\nu'}z_0^{1-k|a_2|}(z')^{\nu'}\\
h_i(z_0, \ldots, z_n) 
&=& \sum_{k=0}^{\infty} \, {\sum_{|\nu'|=1+k|a_1|}}^{\hspace{-13pt}\prime}\,\,\, 
b^{(i)}_{\nu'}z_0^{-k|a_2|}(z')^{\nu'}
\end{eqnarray*}
for $1\leq i \leq n$. We have $b^{(0)}_{(1,0,\ldots,0)} \neq 0$ as above.
We mention the first degree terms of $f \circ h$.
For the first component
\begin{eqnarray*}
f_0(h_0, \ldots, h_n) 
&=& a^{(0)}_{(1,0,\ldots,0)}h_0 + 
\sum_{k=1}^{\infty} \, {\sum_{|\nu'|=k|a_1|}}^{\hspace{-7pt}\prime}\,\,\,\,\,\,
a^{(0)}_{\nu'}h_0^{1-k|a_2|}(h')^{\nu'}.
\end{eqnarray*}
Then, for $k>0$,
\begin{eqnarray*}
h_0(z)^{1-k|a_2|} &=& \bigl(\sum_{l=0}^{\infty} \, {\sum_{|\nu'|=l|a_1|}}^{\hspace{-8pt}\prime}
b^{(0)}_{\nu'}z_0^{1-l|a_2|}(z')^{\nu'}\bigr)^{1-k|a_2|}
=
z_0^{1-k|a_2|} \bigl(\sum_{l=0}^{\infty} \, {\sum_{|\nu'|=l|a_1|}}^{\hspace{-8pt}\prime} 
b^{(0)}_{\nu'}z_0^{-l|a_2|}(z')^{\nu'}\bigr)^{1-k|a_2|}\\
&=&
(b^{(0)}_{0_n}z_0)^{1-k|a_2|} \left(1 - 
\frac{1-k|a_2|}{b^{(0)}_{0_n}} z_0^{-|a_2|} {\sum_{|\nu'|=|a_1|}}^{\hspace{-8pt}\prime}b^{(0)}_{\nu'}(z')^{\nu'}
+ \cdots \right)
\end{eqnarray*}
Thus
$h_0(z)^{1-k|a_2|}$ has the maximum degree of $z_0$ at most $1 - k|a_2| < 1$
and has the minimum degree of $z'$ at least $|a_1| > 1$
in its Laurent expansion.
For $|\nu'|=k|a_1|$ and $k>0$,
$(h')^{\nu'}$ has the maximum degree of $z_0$ at most $-|a_2| < 0$
and
the first degree terms of $z'$ are with coefficients of a negative degree $z_0$ term
in its Laurent expansion.
Hence the first degree term of Laurent expansion of $f_0(h_0, \ldots, h_n)$ 
is $a^{(0)}_{(1,0,\ldots,0)}b^{(0)}_{(1,0,\ldots,0)}z_0$.

Similarly, consider
\begin{equation*}
f_i(h_0, \ldots, h_n) 
= {\sum_{|\nu'|=1}}^{\prime}\,\, a^{(i)}_{\nu'}(h')^{\nu'} + 
\sum_{k=1}^{\infty} \, 
{\sum_{|\nu'|=1+k|a_1|}}^{\hspace{-13pt}\prime}\,\,a^{(i)}_{\nu'}h_0^{-k|a_2|}(h')^{\nu'}
\end{equation*}
for $1\leq i \leq n$.
Then, for $k>0$,
\begin{eqnarray*}
h_0^{-k|a_2|}
&=& (b^{(0)}_{0_n}z_0)^{-k|a_2|} \left(1 - 
\frac{-k|a_2|}{b^{(0)}_{0_n}} z_0^{-|a_2|} {\sum_{|\nu'|=|a_1|}}^{\hspace{-8pt}\prime}b^{(0)}_{\nu'}(z')^{\nu'}
+ \cdots \right).
\end{eqnarray*}
Thus $h_0^{-k|a_2|}$ has the maximum degree of $z_0$ at most $- k|a_2| < 0$
 and has the minimum degree of $z'$ at least $|a_1| > 1$
 in its Laurent expansion.
For $|\nu'|=1+k|a_1|$ and $k>0$, $(h')^{\nu'}$
has the maximum degree of $z_0$ at most $-|a_2| < 0$
and the first degree terms of $z'$ are with coefficients of negative degree $z_0$ term 
in its Laurent expansion.
Hence the first degree terms of the Laurent expansions of $f_i(h_0, \ldots, h_n)$ 
is 
\begin{eqnarray*}
\sum_{j=1}^n {\sum_{|\nu'|=1}}^{\prime}\, a^{(i)}_{\nu_j}b^{(j)}_{\nu'}(z')^{\nu'},
\end{eqnarray*}
where $\nu_j=(0,\ldots,0,1_j,0,\ldots,0)$,
that is, the $j$-th component is $1$ and the others are $0$.

We put $Pf$ as (\ref{Pf}).
Consequently, 
\begin{eqnarray*}\label{}
P(f \circ h) = Pf \circ Ph, \,\,\,{\rm and}\,\,\, P{\rm id} = {\rm id},
\end{eqnarray*}
and therefore 
\begin{eqnarray*}\label{}
Pf  \in GL(n+1, \mathbb{C})
\end{eqnarray*}
since $f$ is an automorphism.
Then the same argument as that in previous case ($\lambda < 0$) shows that
this is a contradiction.
Thus it does not occur that $\lambda$ is positive non-integer.

Hence
we have $\lambda = c_2/c_1 = a_2/a_1 \in \mathbb{Z}\setminus \{0\}$ and 
$a_1=\pm1$.
We now prove $\lambda = \pm 1$.
By (\ref{ac0}), (\ref{aci}) and Remark \ref{remt}, 
the Laurent expansions of $f \in \mathrm{Aut}(\Omega)$ are
\begin{eqnarray*}
f_0(z_0, \ldots, z_n) 
= \sum_{k=0}^{\infty} \, {\sum_{|\nu'|=k}}' 
a^{(0)}_{\nu'}z_0^{1-k\lambda}(z')^{\nu'},
\end{eqnarray*}
and
\begin{eqnarray*}
f_i(z_0, \ldots, z_n) 
= a^{(i)}_{(\lambda,0,\ldots,0)}z_0^{\lambda} + 
\sum_{k=0}^{\infty} \, {\sum_{|\nu'|=1+k}}^{\hspace{-7pt}\prime} \,\,
a^{(i)}_{\nu'}z_0^{-k\lambda}(z')^{\nu'}
\end{eqnarray*}
for $1\leq i \leq n$.
Consider the actions of $(e^{2\pi i\frac{m}{\lambda}},1, \ldots, 1) \in T^{n+1}$
on $\Omega$, for $1 \leq m \leq |\lambda|$.
Then
\begin{eqnarray*}
f_0(e^{2\pi i\frac{m}{\lambda}}z_0, \ldots, z_n) 
&=& \sum_{k=0}^{\infty} \, {\sum_{|\nu'|=k}}' 
a^{(0)}_{\nu'}(e^{2\pi i\frac{m}{\lambda}}z_0)^{1-k\lambda}(z')^{\nu'}\\
&=& e^{2\pi i\frac{m}{\lambda}}\sum_{k=0}^{\infty} \, {\sum_{|\nu'|=k}}' 
a^{(0)}_{\nu'}z_0^{1-k\lambda}(z')^{\nu'}\\
&=& e^{2\pi i\frac{m}{\lambda}} f_0(z_0, \ldots, z_n),
\end{eqnarray*}
and
\begin{eqnarray*}
f_i(e^{2\pi i\frac{m}{\lambda}}z_0, \ldots, z_n) 
= f_i(z_0, \ldots, z_n) 
\end{eqnarray*}
for $1\leq i \leq n$.
Thus $(e^{2\pi i\frac{m}{\lambda}},1, \ldots, 1) \in T^{n+1}$ for $1 \leq m \leq |\lambda|$
are included in the center $\rho(\mathbb{C}^*)$
of $\rho(GU(n,1))$.
Since $a_2c_2 \neq 0$, we see that the integer $\lambda$ must be $\pm 1$.

\bigskip
{\bf Case (ii)}: $c_1 \neq 0$, $c_2 = 0$.\\
In this case,
$\Omega \subset  \mathbb{C}^{n+1}$ can be written
of the form $(\mathbb{C} \times D) \cup  (\mathbb{C}^* \times D')$,
where $D$ and $D'$ are open sets in $\mathbb{C}^n$.
Indeed, 
$\Omega = (\Omega \cap \{z_0 = 0\}) \cup (\Omega \cap \{z_0 \neq 0\})$.
Then $\{0\} \times D := \Omega \cap \{z_0 = 0\} \subset \Omega$
implies $\mathbb{C} \times D \subset \Omega$ 
by $\rho(\mathbb{C}^*)$- and $T^{n+1}$-actions on $\Omega$.
On the other hand,
$\Omega \cap \{z_0 \neq 0\} = \mathbb{C}^* \times D'$
for some open set $D' \subset \mathbb{C}^n$ 
by $\rho(\mathbb{C}^*)$- and $T^{n+1}$-actions.
Thus 
$\Omega = (\mathbb{C} \times D) \cup  (\mathbb{C}^* \times D')$.
Then, by Lemma \ref{notLie}, $\mathrm{Aut}(\Omega)$ has no Lie group structure,
and this contradicts the assumption $\mathrm{Aut}(\Omega) = GU(n,1)$.

\bigskip
{\bf Case (iii)}: $c_1 = 0$ and $c_2 \neq 0$.\\
As in the previous case, $\Omega \subset  \mathbb{C}^{n+1}$ can be written
of the form $(D'' \times \mathbb{C}^n) \cup  (D''' \times \mathbb{C}^n \setminus \{0\} )$
by $\rho(\mathbb{C}^*)$- and $T^{n+1}$-actions on $\Omega$,
where $D''$ and $D'''$ are open sets in $\mathbb{C}$.
Then, by Lemma \ref{notLie}, 
$\mathrm{Aut}(\Omega)$ has no Lie group structure,
and this contradicts our assumption. 
\end{proof}

\bigskip
Since $G=\rho(G(U(1) \times U(n)))$ acts as linear transformations on 
$\Omega \subset \mathbb{C}^{n+1}$, 
it preserves the boundary $\partial \Omega$ of $\Omega$.
We now study the action of $G$ on $\partial \Omega$.
$G$-orbits of points in $\mathbb{C}^{n+1}$ 
consist of four types as follows:

(i) If 
$p = (p_0, p_1, \ldots, p_n) \in \mathbb{C}^* \times (\mathbb{C}^n \setminus \{0_n\})$,
then
\begin{equation} \label{GU1}
G \cdot p = \{ (z_0, \ldots, z_n) \in \mathbb{C}^{n+1} \setminus \{0\}
 : -a|z_0|^{2\lambda} + |z_1|^2 + \cdots + |z_n|^2 = 0\},
\end{equation}
where
$a:=(|p_1|^2 + \cdots + |p_n|^2)/|p_0|^{2\lambda} > 0$
and $\lambda = \pm 1$ by Claim \ref{cla1}.

\smallskip

(ii) If $p' = (0, p'_1, \ldots, p'_n) \in \mathbb{C}^{n+1} \setminus \{0\}$,
then
\begin{equation} \label{GU2}
G \cdot p' = \{0_1\} \times (\mathbb{C}^n \setminus \{0_n\}).
\end{equation}

\smallskip

(iii) If $p'' = (p''_0, 0, \ldots, 0) \in \mathbb{C}^{n+1} \setminus \{0\}$,
then
\begin{equation} \label{GU3}
G \cdot p'' = \mathbb{C}^* \times \{0_n\}.
\end{equation}

\smallskip

(iv) If $p''' = (0, \ldots, 0) \in \mathbb{C}^{n+1}$,
then
\begin{equation} \label{GU4}
G \cdot p''' = \{0\} \subset \mathbb{C}^{n+1}.
\end{equation}

\smallskip

\begin{claim}\label{5.1}
$\Omega \cap (\mathbb{C}^* \times (\mathbb{C}^n \setminus \{0_n\}))$ is a proper subset of
$\mathbb{C}^* \times (\mathbb{C}^n \setminus \{0_n\})$.
\end{claim}

\begin{proof}
If $\Omega \cap (\mathbb{C}^* \times (\mathbb{C}^n \setminus \{0_n\}))
= \mathbb{C}^* \times (\mathbb{C}^n \setminus \{0_n\})$,
then 
$\Omega$ equals one of the following domains 
by $G$-actions of type (\ref{GU2}) and (\ref{GU3}) above:
\begin{equation*} 
\mathbb{C}^{n+1},
\mathbb{C}^{n+1}  \setminus \{0\},
\mathbb{C}^* \times (\mathbb{C}^n \setminus \{0_n\}),
\mathbb{C} \times (\mathbb{C}^n \setminus \{0_n\})
\,\,\mathrm{or} \,\,\mathbb{C}^* \times \mathbb{C}^n.
\end{equation*}
However these can not occur 
since all automorphism groups of these domains 
are not Lie groups, by Lemma \ref{notLie}.
This contradicts that
$\mathrm{Aut}(\Omega) = GU(n,1)$.
\end{proof}

\bigskip
By Claim \ref{5.1}, 
$\partial \Omega \cap (\mathbb{C}^* \times (\mathbb{C}^n \setminus \{0_n\}))
\neq \emptyset$.
Thus we can take a point 
\begin{eqnarray*} 
p=(p_0, \ldots, p_n) \in \partial \Omega \cap 
(\mathbb{C}^* \times (\mathbb{C}^n \setminus \{0_n\})).
\end{eqnarray*} 
Let 
\begin{eqnarray*} 
&&a=(|p_1|^2 + \cdots + |p_n|^2)/|p_0|^{2\lambda} > 0,\\
&&A_{a,\lambda} = \{ (z_0, \ldots, z_n) \in \mathbb{C}^{n+1} : 
-a|z_0|^{2\lambda} + |z_1|^2 + \cdots + |z_n|^2 = 0\}.
\end{eqnarray*}
Note that 
\begin{equation*} 
\partial \Omega  \supset A_{a,\lambda}.
\end{equation*}
If $\lambda = 1$,
then $\Omega$ is included in
\begin{equation*} 
D_{a,1} = \{ |z_1|^2 + \cdots + |z_n|^2 > a|z_0|^{2}\}
\end{equation*}
or 
\begin{equation*} 
C_{a,1} = \{ |z_1|^2 + \cdots + |z_n|^2 < a|z_0|^{2}\}.
\end{equation*}
If $\lambda = -1$,
then $\Omega$ is included in
\begin{equation*} 
D_{a,-1} = \{ (|z_1|^2 + \cdots + |z_n|^2)|z_0|^{2} > a\}
\end{equation*}
or 
\begin{equation*} 
C_{a,-1} = \{ (|z_1|^2 + \cdots + |z_n|^2)|z_0|^{2} < a\}.
\end{equation*}

\begin{claim}\label{eqthm}
If $\Omega=D_{a,1}$,  
then $\Omega$ is biholomorphic to $D^{n,1}$.
\end{claim}

\begin{proof}
Indeed there exists a biholomorphic map
\begin{eqnarray*}
\Phi : D_{a,1} \ni (z_0, z_1, \ldots, z_n) \longmapsto (a^{-1/2}z_0, z_1, \ldots, z_n) \in D^{n,1}.
\end{eqnarray*}
\end{proof}

We will show that Claim \ref{eqthm} is the only case that a domain has the automorphism 
group isomorphic to $GU(n,1)$.

Let us first consider the case $\partial \Omega = A_{a,\lambda}$,
that is, $\Omega = C_{a,1}, D_{a,-1}$ or $C_{a,-1}$,
and we derive contradictions.

\begin{claim}
$\mathrm{Aut}(C_{a,1})$ and $\mathrm{Aut}(D_{a,-1})$
are not  Lie groups,  
so $\Omega \neq C_{a,1}, D_{a,-1}$.
\end{claim}

\begin{proof}
Indeed, $C_{a,1}$ is biholomorphic to
$\mathbb{C}^* \times \mathbb{B}^n$,
and 
$D_{a,-1}$ is biholomorphic to
$\mathbb{C}^* \times (\mathbb{C}^n \setminus \mathbb{B}^n)$.
The automorphism groups of these domains are not Lie groups,
by Lemma \ref{notLie}.
\end{proof}

\begin{claim} \label{cla4}
$\Omega \neq C_{a,-1}$.
\end{claim}

\begin{proof}
Suppose $\Omega = C_{a,-1}$.
Then, for $f \in \mathrm{Aut}(\Omega)\setminus G$,
the Laurent expansions are
\begin{eqnarray*}
f_0(z_0, \ldots, z_n) 
=  \sum_{k=0}^{\infty} {\sum_{|\nu'|=k}}' 
a^{(0)}_{\nu'}z_0^{1+k}(z')^{\nu'},
\end{eqnarray*}
and
\begin{eqnarray*}
f_i(z_0, \ldots, z_n) = 
a^{(i)}_{(\lambda,0,\ldots,0)}z_0^{-1} + 
 \sum_{k=0}^{\infty} {\sum_{|\nu'|=1+k}}^{\hspace{-7pt}\prime}\,\, 
a^{(i)}_{\nu'}z_0^{k}(z')^{\nu'}
\end{eqnarray*}
for $1\leq i \leq n$.
Since $C_{a,-1} \cap \{z_0=0\} \neq \emptyset$,
negative degree of $z_0$ does not arise in the Laurent expansions.
Therefore
\begin{eqnarray*}
f_0(z_0, \ldots, z_n) 
&=&  \sum_{k=0}^{\infty}  {\sum_{|\nu'|=k}}' 
a^{(0)}_{\nu'}z_0^{1+k}(z')^{\nu'},
\\
f_i(z_0, \ldots, z_n) 
&=&  \sum_{k=0}^{\infty} {\sum_{|\nu'|=1+k}}^{\hspace{-7pt}\prime}\,\,
a^{(i)}_{\nu'}z_0^{k}(z')^{\nu'}
\end{eqnarray*}
for $1\leq i \leq n$.
Consider 
\begin{eqnarray*}
Pf (z) = (a^{(0)}_{(1,0,\ldots,0)}z_0, {\sum_{|\nu'|=1}}' a^{(1)}_{(0,\nu')}(z')^{\nu'}, \ldots,
{\sum_{|\nu'|=1}}' a^{(n)}_{(0,\nu')}(z')^{\nu'}),
\end{eqnarray*}
as in the proof of Claim \ref{cla1}.
Then as a matrix we can write
\begin{eqnarray*}
Pf = 
\begin{pmatrix}
a^{(0)}_{(1,0,\ldots,0)} & 0 & \cdots & 0 &\\
0 & a^{(1)}_{(0,1,0,\ldots,0)} & \cdots & a^{(1)}_{(0,\ldots,0,1)}\\
 \vdots  & \vdots &  \ddots & \vdots\\
0 & a^{(n)}_{(0,1,0,\ldots,0)} & \cdots           & a^{(n)}_{(0,\ldots,0,1)}
\end{pmatrix}.
\end{eqnarray*}
Then it follows that
\begin{eqnarray*}\label{}
P(f \circ h) = Pf \circ Ph, \,\,\,{\rm and}\,\,\, P{\rm id} = {\rm id},
\end{eqnarray*}
where $h \in \mathrm{Aut}(\Omega)$,
and therefore 
\begin{eqnarray*}\label{}
Pf  \in GL(n+1, \mathbb{C})
\end{eqnarray*}
since $f$ is an automorphism.
Hence we have a representation of $GU(n,1)$ given by
\begin{eqnarray*}
GU(n,1) \ni g \longmapsto Pf \in GL(n+1,\mathbb{C}),
\end{eqnarray*}
where $f = \rho(g)$.
The restriction of this representation to the simple Lie group $SU(n,1)$
is nontrivial since $\rho(U(1) \times U(n)) = U(1) \times U(n)$.
However this contradicts Lemma \ref{notsimple}.
Thus $\Omega \neq C_{a,-1}$.
\end{proof}

\bigskip
Let us consider the case $\partial \Omega \neq A_{a,\lambda}$.

\smallskip

{\bf Case (I)} : $(\partial \Omega \setminus A_{a,\lambda}) \cap (\mathbb{C}^* \times (\mathbb{C}^n \setminus \{0_n\}))
 = \emptyset$.
\\
In this case,
$\partial \Omega$ is the union of $A_{a,\lambda}$ and some of the following sets
\begin{equation}\label{set}
\{0_1\} \times (\mathbb{C}^n \setminus \{0_n\}),
\mathbb{C}^* \times \{0_n\}
\,\,\mathrm{or} \,\,
\{0\} \subset \mathbb{C}^{n+1},
\end{equation}
by the $G$-actions on the boundary of type 
(\ref{GU2}), (\ref{GU3}) and (\ref{GU4}) above.
If $\Omega \subset D_{a,-1}$,
then sets in (\ref{set})
can not be included in the boundary of $\Omega$.
Thus we must consider only the case $\Omega \subsetneq D_{a,1}$,
$C_{a,1}$ or $C_{a,-1}$.

\smallskip

{\bf Case (I-i)} : $\Omega \subsetneq D_{a,1}$.\\
In this case, $\mathbb{C}^* \times \{0_n\}$ can not be a subset of the boundary of $\Omega$,
and $\{0\} \in A_{a,1}$.
Thus 
\begin{eqnarray*}
&&\partial \Omega = A_{a,1} \cup (\{0_1\} \times \mathbb{C}^n),\\
&&\Omega = D_{a,1} \setminus (\{0_1\} \times \mathbb{C}^n).
\end{eqnarray*}
Then, $\Omega$ is biholomorphic to $\mathbb{C}^* \times (\mathbb{C}^n \setminus
\mathbb{B}^{n})$ and 
$\mathrm{Aut}(\mathbb{C}^* \times (\mathbb{C}^n \setminus \mathbb{B}^{n}))$
does not have a Lie group structure.
This contradicts the assumption that $\mathrm{Aut}(\Omega) = GU(n,1)$.
Thus this case does not occur.

\smallskip

{\bf Case (I-ii)} : $\Omega \subsetneq C_{a,1}$.\\
In this case, $\{0\} \times (\mathbb{C}^n \setminus \{0_n\})$ 
can not be a subset of the boundary of $\Omega$,
and $\{0\} \in A_{a,1}$.
Thus 
\begin{eqnarray*}
&&\partial \Omega = A_{a,1} \cup (\mathbb{C} \times \{0_n\}),\\
&&\Omega = C_{a,1} \setminus (\mathbb{C} \times \{0_n\}).
\end{eqnarray*}
Then, $\Omega$ is biholomorphic to 
$\mathbb{C}^* \times (\mathbb{B}^{n} \setminus \{0_n\})$
and 
$\mathrm{Aut}(\mathbb{C}^* \times (\mathbb{B}^{n} \setminus \{0_n\}))$
does not have a 
Lie group structure.
This contradicts the assumption that $\mathrm{Aut}(\Omega) = GU(n,1)$,
and this case does not occur.

\smallskip

{\bf Case (I-iii)} : $\Omega \subsetneq C_{a,-1}$.\\
In this case, $\Omega$ coincides with one of the followings:
\begin{eqnarray*}
&&C_1=C_{a,-1} \setminus (\{0_1\} \times \mathbb{C}^n) \cup(\mathbb{C} \times \{0_n\}),
\\
&&C_2=C_{a,-1} \setminus (\{0_1\} \times \mathbb{C}^n),\\
&&C_3=C_{a,-1} \setminus (\mathbb{C} \times \{0_n\}),\\
&&C_4=C_{a,-1} \setminus \{0_{n+1}\}.
\end{eqnarray*}
Then
$C_1$ is biholomorphic to $\mathbb{C}^* \times (\mathbb{B}^n \setminus \{0_n\})$,
and $C_2$ is biholomorphic to $\mathbb{C}^* \times \mathbb{B}^n$.
The automorphism groups of these domains are not Lie groups.
This contradicts the assumption.
The proof of Claim \ref{cla4} also leads that $\Omega \neq C_3, C_4$
since $C_3 \cap \{z_0=0\} \neq \emptyset$ and $C_4 \cap \{z_0=0\} \neq \emptyset$.
Thus this case does not occur.

\bigskip

{\bf Case (I\hspace{-.1em}I)} : 
$(\partial \Omega \setminus A_{a,\lambda}) \cap (\mathbb{C}^* \times (\mathbb{C}^n \setminus \{0_n\})) \neq \emptyset$.
\\
In this case, we can take a point 
$p'=(p'_0, \ldots, p'_n) \in (\partial \Omega \setminus A_{a,\lambda}) \cap (\mathbb{C}^* \times (\mathbb{C}^n \setminus \{0_n\}))$.
Put 
\begin{eqnarray*}
&&b=(|p'_1|^2 + \cdots + |p'_n|^2)/|p'_0|^{2\lambda} > 0,\\
&&B_{b,\lambda}= \{(z_0, \ldots, z_n) \in \mathbb{C}^{n+1} : -b|z_0|^{2\lambda} + |z_1|^2 + \cdots + |z_n|^2 = 0\}.
\end{eqnarray*}
We may assume $a>b$ without loss of generality.

\smallskip

{\bf Case (I\hspace{-.1em}I-i)} : $\partial \Omega = A_{a,\lambda} \cup B_{b,\lambda}$. \\
Since $\Omega$ is connected, it coincides with
\begin{eqnarray*} 
C_{a,1} \cap D_{b,1} 
= \{ b|z_0|^{2} < |z_1|^2 + \cdots + |z_n|^2 < a|z_0|^{2} \},
\end{eqnarray*}
or
\begin{eqnarray*} 
C_{a,-1} \cap D^{-}_{b,-1} 
= \{ b < (|z_1|^2 + \cdots + |z_n|^2)|z_0|^{2} < a \}.
\end{eqnarray*}
These domains are biholomorphic to $\mathbb{C}^* \times \mathbb{B}^n(a,b)$, 
where 
\begin{equation*} 
\mathbb{B}^n(a,b) = \{(z_1, \ldots, z_n) \in \mathbb{C}^n : b<|z_{1}|^{2} + \cdots + |z_{n}|^{2} < a\}.
\end{equation*}
Then $\mathrm{Aut}(\mathbb{C}^* \times \mathbb{B}^n(a,b))$
does not have a Lie group structure by Lemma \ref{notLie}, 
and this contradicts our assumption.
Thus this case does not occur.

\smallskip

{\bf Case (I\hspace{-.1em}I-ii)} : 
$\partial \Omega \neq A_{a,\lambda} \cup B_{b,\lambda}$.
\\
Suppose $\partial \Omega \cap (\mathbb{C}^* \times \mathbb{C}^n \setminus \{0_n\})
\setminus (A_{a,\lambda} \cup B_{b,\lambda}) \neq \emptyset$,
then we can take 
\begin{equation*} 
p''=(p''_0, \ldots, p''_n) \in \partial \Omega \cap (\mathbb{C}^* \times \mathbb{C}^n 
\setminus \{0_n\}) \setminus (A_{a,\lambda} \cup B_{b,\lambda}).
\end{equation*}
Then put
\begin{eqnarray*} 
&&c=(|p''_1|^2 + \cdots + |p''_n|^2)/|p''_0|^{2\lambda},\\
&&C_{c,\lambda} = \{(z_0, \ldots, z_n) \in \mathbb{C}^{n+1} : 
-c|z_0|^{2\lambda} + |z_1|^2 + \cdots + |z_n|^2 = 0\}.
\end{eqnarray*}
We have $A_{a,\lambda} \cup B_{b,\lambda} \cup C_{c,\lambda} \subset \partial \Omega$.
However $\Omega$ is connected, this is impossible. 
Therefore this case does not occur.
Let us consider the remaining case:
\begin{eqnarray*} 
\partial \Omega \cap (\mathbb{C}^* \times \mathbb{C}^n \setminus \{0_n\})
\setminus (A_{a,\lambda} \cup B_{b,\lambda}) = \emptyset.
\end{eqnarray*}
However,
$\mathbb{C}^* \times \{0_n\}$, $\{0\} \times (\mathbb{C}^n \setminus \{0_n\})$
and $\{0\} \in \mathbb{C}^{n+1}$
can not be subsets of the boundary of $\Omega$
since $\Omega \subset C_{a,1} \cap D_{b,1}$
or
$\Omega \subset C_{a,-1} \cap D_{b,-1}$.
Thus this case does not occur either.

\smallskip
We have shown that $\partial \Omega = A_{a,1}$ and $\Omega = D_{a,1}$ 
which is biholomorphic to $D^{n,1}$.
\end{proof}

\bigskip

\section{A counterexample of the group-theoretic characterization}

\begin{theorem}\label{fth}
There exist unbounded homogeneous domains in $\mathbb{C}^n, n \geq 5$ which are not biholomorphically equivalent,
while their automorphism groups are isomorphic.
\end{theorem}

\begin{proof}
Suppose $p,q>1$ and $p+q=n$.
Let
\begin{eqnarray*}
&&D^{p,q}=\{ (z_1, \ldots, z_p, w_1, \ldots, w_q) \in \mathbb{C}^{n} : 
 |z_{1}|^{2}+ \cdots + |z_{p}|^{2} -|w_{1}|^{2} -\cdots -|w_q|^2 > 0 \},\\
&&C^{p,q}=\{ (z_1, \ldots, z_p, w_1, \ldots, w_q) \in \mathbb{C}^{n} : 
 |z_{1}|^{2}+ \cdots + |z_{p}|^{2} -|w_{1}|^{2} -\cdots -|w_q|^2 < 0 \}.
\end{eqnarray*}
If $p \neq q$, then $D^{p,q}$ and $C^{p,q}$ are not biholomorphically equivalent, while
$\mathrm{Aut}(D^{p,q}) = \mathrm{Aut}(C^{p,q})$.
Indeed,
as the proof of Theorem \ref{1}, 
we take $f \in \mathrm{Aut}(D^{p,q})$.
If $(w'_1, \ldots, w'_q) \in \mathbb{C}^q$ is fixed, then the holomorphic functions 
$f_i(\cdots, w'_1, \ldots, w'_q)$ for $i=1, \ldots, n$, on 
$D^{p,q} \cap \{ w_1 = w'_1, \ldots, w_q = w'_q \}$ extend holomorphically to 
$\mathbb{C}^{n} \cap \{ w_1 = w'_1, \ldots, w_q = w'_q \}$ by Hartogs theorem and $p>1$.
Hence, when $w_1, \ldots, w_q$ vary, we obtain a extended holomorphic map $\tilde{f} : \mathbb{C}^{n} \longrightarrow \mathbb{C}^{n}$ such that 
$\tilde{f}|_{D^{p,q}} = f \in \mathrm{Aut}(D^{p,q})$.
The same consideration for $f^{-1} \in \mathrm{Aut}(D^{p,q})$ 
shows that there exists a holomorphic map $g : \mathbb{C}^{n+1} \longrightarrow \mathbb{C}^{n+1}$, 
such that
$g|_{D^{p,q}} = f^{-1}$.
Since $g\circ f = \mathrm{id}$ and $f\circ g = \mathrm{id}$ on $D^{p,q}$, 
the uniqueness of analytic continuation shows that
$g\circ \tilde{f} = \mathrm{id}$ and $\tilde{f}\circ g = \mathrm{id}$ on $\mathbb{C}^{n}$.
Hence $\tilde{f} \in \mathrm{Aut(\mathbb{C}^{n})}$.
Now we see that $\tilde{f}|_{C^{p,q}} \in \mathrm{Aut}(C^{p,q})$ and therefore
we have a group homomorphism
\begin{eqnarray*}
\phi : \mathrm{Aut}(D^{p,q}) \longrightarrow \mathrm{Aut}(C^{p,q}),
\,\,\,\,f \longmapsto \tilde{f}|_{C^{p,q}}.
\end{eqnarray*}

In the same manner, we have
\begin{eqnarray*}
\psi : \mathrm{Aut}(C^{p,q}) \longrightarrow \mathrm{Aut}(D^{p,q}),
\,\,\,\,g \longmapsto \tilde{g}|_{C^{p,q}}.
\end{eqnarray*}
by Hartogs theorem and $q>1$.
It is clear that $\phi \circ \psi = \mathrm{id}$ on $\mathrm{Aut}(C^{p,q})$ and 
$\psi \circ \phi = \mathrm{id}$ on $\mathrm{Aut}(D^{p,q})$.
Thus we obtain $\mathrm{Aut}(D^{p,q}) \simeq \mathrm{Aut}(C^{p,q})$.
\end{proof}

\bigskip
We have not yet 
obtained a explicit description of the automorphism groups $\mathrm{Aut}(D^{p,q})$
for $p, q >1$.
We only expect that $\mathrm{Aut}(D^{p,q}) = GU(p,q)$,
where
\begin{equation*}
GU(p,q)=\{M \in GL(n, \mathbb{C}):M^*JM = \nu(M)J, \mathrm{for\,some}\, \nu(M) \in \mathbb{R}_{>0} \},
\end{equation*}
and 
$J=
\begin{pmatrix}
E_p &   0\\
0 &  -E_q
\end{pmatrix}
$.

The difference between $D^{n,1}$ and $D^{p,q}$ for $p,q>1$ is that
the exterior of $D^{n,1}$ is holomorphically convex domain,
but that of $D^{p,q}$ is not.
It is known that
some holomorphically convex homogeneous Reinhardt domains
are characterized by its automorphism groups with some additional conditions
(see \cite{BKS} and \cite{KS1}).
We may proceed with the  group-theoretic characterization problem
for holomorphically convex homogeneous Reinhardt domains,
or for homogeneous Reinhardt domains with a holomorphically convex exterior domain.


\bigskip

\end{document}